\documentclass[preprint]{elsarticle}       
\usepackage{lipsum}
\makeatletter
\def\ps@pprintTitle{%
\let\@oddhead\@empty
\let\@evenhead\@empty
\def\@oddfoot{}%
\let\@evenfoot\@oddfoot}
\makeatother
\usepackage{dsfont}
\usepackage{lineno,hyperref}
\usepackage{lscape}
\usepackage{enumitem}
\usepackage{rotating}
\usepackage{makeidx}
\usepackage{float}
\usepackage{epsfig}
\usepackage{amsmath,amssymb}
\usepackage{amsthm}
\usepackage{enumitem}
\usepackage{multirow}
\usepackage[utf8]{inputenc}
\usepackage{epic,eepic}
\usepackage{overpic}
\usepackage{color}
\usepackage{amsfonts}
\usepackage{graphicx}
\usepackage{enumitem}
\usepackage{xcolor}
\usepackage{cancel}
\usepackage{xspace}
\usepackage{bm}

\newcommand{\R}{\mathds{R}}
\newcommand{\Z}{\mathds{Z}}

\newcommand{\lra}{\longrightarrow}

\newcommand{\cC}{\mathcal{C}}

\newcommand{\ba}{\bm{a}}
\newcommand{\bE}{\bm{E}}
\newcommand{\be}{\bm{e}}

\newcommand{\bff}{\bm{f}}

\newcommand{\bx}{\bm{x}}

\newcommand{\bbeta}{\bm{\beta}}

\newcommand{\dif}{\nabla_d}

\newcommand{\var}[1]{\mathbb{V}\left(#1\right)}
\newcommand{\mean}[1]{\mathbb{E}\left(#1\right)}

\theoremstyle{plain}
\newtheorem{thm}{Theorem}
\newtheorem{lem}[thm]{Lemma}

\newtheorem{cor}[thm]{Corollary}
\newtheorem{prop}[thm]{Proposition}

\newtheorem{defi}[thm]{Definition}

\theoremstyle{remark}
\newtheorem{rmk}[thm]{Remark}

\usepackage{mathtools}
\mathtoolsset{showonlyrefs}

\begin{document}

\begin{frontmatter}

	\title{Linear minimum-variance approximants for noisy data}
	
	\author[UV1]{Sergio L\'opez-Ure\~na}
	\ead{sergio.lopez-urena@uv.es}
	\author[UV1]{Dionisio F. Y\'a\~nez}
	\ead{dionisio.yanez@uv.es}
	\address[UV1]{Departament de Matem\`atiques.  Universitat de Val\`encia (EG) (Spain)}
	
	\begin{abstract}
        Inspired by recent developments in subdivision schemes founded on the Weighted Least Squares technique, we construct linear approximants for noisy data in which the weighting strategy minimizes the output variance, thereby establishing a direct correspondence with the Generalized Least Squares and the Minimum-Variance Formulas methodologies. By introducing annihilation-operators for polynomial spaces, we derive usable formulas that are optimal for general correlated non-uniform noise. We show that earlier subdivision rules are optimal for uncorrelated non-uniform noise and, finally, we present numerical evidence to confirm that, in the correlated case, the proposed approximants are better than those currently used in the subdivision literature.
	\end{abstract}
	
	\begin{keyword}
		Linear approximants \sep noisy data \sep variance minimization \sep correlated non-uniform noise.
	\end{keyword}
	
\end{frontmatter}


\section{Introduction}

In certain applications, noisy data are managed, requiring the use of specialized approximants. Many have been proposed in the literature, such as Moving Least Squares (MLS), Kernel Regression methods (see \cite{Levin98,TFM07}) or Generalized Least Squares (GLS) (see \cite{Aitken36}), for instance. 
Our work is motivated by recent advances on subdivision theory on this problem. In particular, in \cite{LUYA} the Weigthed Least Squares method (WLS, that can be considered a particular case of GLS) was used to define new subdivision schemes. We continue with the idea followed by seminal works on Minimum-Variance Formulas (MVF) (v.g. \cite{CR76}), that propose to minimize the approximation variance subject to it is exact for polynomials up to a given degree. The MVF method was proven to be equivalent to the GLS one (see \cite{Aitken36}), thus, the present work can be considered a continuation of \cite{LUYA}.



Let us introduce some notation and concepts: Let $\bff = (f_i)_{i=1}^{N} \in\R^N$ represent a dataset associated with the grid $\bx=(x_i)_{i=1}^{N} \in \R^N$, $x_i < x_{i+1}$. We assume the data has two components: 
noise $\be\in\R^N$ and samples of a smooth function $G:\R \lra \R$. Specifically, $\bff = G|_{\bx}+\be$, where $G|_{\bx} := (G(x_i))_{i=1}^{N}$ and the noise $\be \sim \bE,$ being $\bE$ a vector of probability distributions with zero mean, i.e. $\mean{\bE} = 0\in\R^N.$ Our objective is to determine coefficients $\ba\in\R^N$ for a linear approximant, $\Psi(\bff) := \sum_{i=1}^{N} a_i f_{i}$, such that $\Psi(\bff)\approx G(t_0)$ for a given $t_0\in[x_1,x_N]$, achieving a specified accuracy while minimizing the influence of $\be$. This constrain is equivalent to demand that $\Psi (P|_{\bx}) = P(t_0)$, $\forall P\in\Pi_d$, where $d$ is a fixed polynomial degree, known as \emph{polynomial reproduction} property in subdivision theory (see \cite{Lopez24,dynlevinluzzato}).



We assume that the covariance matrix of $\bE$ is not diagonal, in contrast to studies such as \cite{CR76}. In practical applications, it is often necessary to adopt this generalization, as the assumption of simple noise structures may not hold. To address such more general noise behavior, several approaches have been developed in the time-series literature, including the Autoregressive Integrated Moving Average (ARIMA) models \cite{Shumway25}, the AutoRegressive eXogenous (ARX) models \cite{Maurya22}, and the Generalized Least Squares (GLS) method \cite{Baltagi}. The covariance matrix can be either known, explicitly modeled -as is common in system identification across various engineering disciplines \cite{Maurya22}- or estimated, for instance, in econometrics using feasible Generalized Least Squares (GLS) \cite{Baltagi}. But ARIMA and ARX are suitable for forecasting, rather than interpolation (see \cite{Yu23}). GLS handles properly the interpolation problem by design, and can be used to fit a high degree polynomial if its parameters are properly set. 

Our contributions are the following: First, we introduce and exploit annihilation operators (a usual concept in subdivision theory, see \cite{dynlevinluzzato,Lopez24}) to parametrize the affine space of coefficients that reproduce polynomials up to a prescribed degree; this algebraic parametrization yields to a different description of the fully general noise model allowing arbitrary correlation and non-uniform variance across samples (i.e. non-diagonal covariance matrix for $\bE$), leading to closed-form expressions for the minimum-variance coefficients. Second, we observe that minimum-variance coefficients are the only ones that, at the same time, allow polynomial reproduction, and they themselves are computed from evaluating a polynomial. Third, we show that the subdivision rules proposed in \cite{LUYA} and preceding works are optimal in the minimum-variance sense for uncorrelated noise. Essentially, this is due to the equivalence between MVF and GLS. In a nutshell, this work provides new perspectives and practical tools for designing subdivision-based approximants under realistic noise models.

In Section \ref{linearcase}, we begin by presenting the solution to the proposed problem and establishing its connection to previous research. Then, in Section \ref{numericalexp}, some experiments are performed to demonstrate the effectiveness of the proposed approximants in comparison to other ones. Finally, some conclusions and future work are presented in Section \ref{conclusions}.

\section{Designing linear approximants to minimize noise under prescribed approximation capability}\label{linearcase}

Since $\Psi$ is linear, we have that $ \Psi(G|_{\bx}+\be) = \Psi (G|_{\bx}) + \Psi (\be)$. We propose to split the problem into two parts. One takes care of approximating the smooth function, and the other tries to reduce the noise as much as possible.
On the one hand, we may solve the \emph{smooth problem}
$\Psi (G|_{\bx}) = G(t_0) + O(h^d)$
that is achieved by imposing on $\Psi$ the reproduction of polynomials up to the $(d-1)$-th degree, $\Pi_{d-1}$, as is well known. For this, we introduce the following definitions.
\begin{defi}
    Let $d,N\in\Z$ be such that $0\leq d < N$. Let $\bx\subset\R^N$ be a grid, and $t_0\in[x_1,x_N]$.
	An approximant $\Psi:\R^N\lra\R$ reproduces $\Pi_{d}$ with the setting $(\bx, t_0)$ if $\Psi (P|_{\bx}) = P(t_0), \,\, \forall P\in \Pi_{d}.$
\end{defi}
The invariance of the polynomial space concerning scaling and shifting operations allows us to infer that if $\Pi_{d}$ is reproduced with the setting $(\bx, t_0)$, then it is also reproduced with $(t_0 + h (\bx-t_0), t_0)$ for any $h > 0$. Notice that the stencil $\bx^h := t_0 + h (\bx-t_0)$ resembles $\bx$, as the distances from any stencil point to $t_0$ are uniformly altered by the same factor $h$:
\[
    x^h_i - t_0 = h(x_i-t_0), \qquad \forall i=1,\ldots, N.
\]
Using that $\Pi_d$ is invariant under affine scalings and translations, one can show that if a linear approximant reproduces $\Pi_{d}$ with the setting $(\bx, t_0)$, then it also reproduces $\Pi_{d}$ with the setting $(\bx^h, t_0)$, where $\bx^h := t_0 + h (\bx-t_0)$, for any $h > 0$. Notice that the affine transformation $x \mapsto t_0 + h (x-t_0)$ fixes $t_0$ and uniformly scales the grid spacings.

\begin{defi}
	An approximant $\Psi:\R^N\lra\R$ has an approximation order $d$ with the setting $(\bx, t_0)$, $t_0\in[x_1,x_N]$, if $\Psi (G|_{\bx^h}) = G(t_0) + O(h^d)$ is fulfilled for any function $G\in \mathcal{C}^d$. That is, if for each $G\in \mathcal{C}^d$ there exist $C_G,h_G>0$ such that $|\Psi (G|_{\bx^h}) - G(t_0)| \leq C_G h^d$ for all $h\in(0,h_G)$.
\end{defi}
\begin{lem}
	If $\Psi$ reproduces $\Pi_{d-1}$ with the setting $(\bx, t_0)$, then it has an approximation order $d$ with the setting $(\bx, t_0)$.
\end{lem}
From now on, we omit $(\bx, t_0)$ if it is clear from the context.

On the other hand, for the \emph{noisy problem} that consists of minimizing the noisy contribution $\Psi (\be)$, we can perform an easy statistical treatment.

\begin{lem}
	Let $\Psi:\R^N\lra\R$ be an approximant defined as $\Psi(\bff) := \ba^T\bff=  \sum_{i=1}^{N} a_i f_{i}$ then
	the expectancy and variance of $\Psi(\bE)$ are
	\begin{equation} \label{eq:error_variance}
		\mean{\Psi(\bE)} = 0, \qquad \var{\Psi(\bE)} = \| \Omega \ba\|_2^2,
	\end{equation}
	where $\Omega$ is some upper triangular matrix.
\end{lem}
\begin{proof}
	The probability distribution of $\Psi (\be)$ is $ \Psi(\bE) = \sum_{i=1}^{N} a_i E_{i} = \ba^T \bE$, denoting the vectors as column matrices. The expectancy and variance are
	$$ \mean{\Psi(\bE)} = \ba^T \mean{\bE} = 0, \qquad \var{\Psi(\bE)} = \ba^T \hat \Omega \ba,$$
	where
	$\hat \Omega$ is the covariance $N\times N$ matrix of the random distributions $\bE$,
	$$ \hat \Omega_{i_1,i_2} = \text{Cov}(E_{i_1},E_{i_2}), \qquad i_1,i_2 = 1,\ldots, N.$$
	Since $\hat \Omega$ is a symmetric and positive semidefinite matrix, it admits a Cholesky-like decomposition, say $ \hat \Omega = \Omega^T \Omega$,
	being $\Omega$ an upper triangular matrix which is invertible, provided that the covariance matrix $\hat \Omega$ is positive definite.
	Now we can write
	$ \var{\Psi(\bE)} = \ba^T \hat \Omega \ba = (\Omega \ba)^T (\Omega \ba) = \| \Omega \ba\|_2^2.$
\end{proof}

Our objective is to minimize the noise level. To this end, we try to reduce the variance of $\Psi (\be)$,  which is equivalent to minimizing  $\|\Omega \ba\|_2^2$.
Accordingly, we formulate the following optimization problem to simultaneously address both the accuracy and noise-related aspects:
\begin{equation} \label{eq:constrained_optimization}
	\begin{cases}
		\min \|\Omega \ba\|_2^2\\
		\text{s.t. $\ba\in\R^{N\times 1}$ such that $\Psi(\bff) = \ba^T \bff$ reproduces $\Pi_{d-1}$ with the setting $(\bx, t_0)$.}
	\end{cases}
\end{equation}

\begin{rmk}
	From now on, we assume that $\Omega$ is invertible. That is, the covariance matrix of the noise is positive-definite.
	Further simplifications can be made, which we do not consider. For instance, if all the variables $\bE$ are \emph{uncorrelated}, then $\Omega$ is just a diagonal matrix. In addition, if all the variables $\bE$ have equal variance, then $\Omega$ is a constant diagonal matrix and it could be assumed without loss of generality that it is the identity matrix. These assumptions are very common in practice, as in the numerical experiments of \cite{AY13,DHHS15,LUYA}.
\end{rmk}

\begin{defi}
	A linear approximant $\Psi$ is said to be \emph{minimum-variance} for a given definite-positive matrix $\hat \Omega$, given the real data $t_0\in\R,\bx\in\R^N$ and a given integer $d\geq 0$, if its coefficient vector $\ba \in \mathbb{R}^N$ is a solution to the constrained optimization problem given in equation \eqref{eq:constrained_optimization}.
\end{defi}

In the following, we compute the minimum-variance approximants for any arbitrary choice of $t_0,d,\bx$ and $\bE$, but supposing that the covariance matrix is positive definite, because it guarantees the unicity of the solution of \eqref{eq:constrained_optimization}.

\begin{lem} \label{lema:reproduction}
$\Psi$ reproduces $\Pi_{d-1}$ if and only if $\ba$ is a solution of the $d\times N$ linear system
	\begin{equation} \label{eq:reproduction_equations}
		\sum_{i=1}^{N} a_i x_i^s = t_0^s, \qquad s = 0,1,\ldots,d-1.
	\end{equation}
	The system can be solved if $N\geq d$, and the solution is unique if $N=d$.
	
	As a consequence, the set of approximants able to reproduce $\Pi_{d-1}$ is an affine space of dimension $N-d$.
\end{lem}
\begin{proof}
	Equation \eqref{eq:reproduction_equations} is obtained from
$\Psi (P|_{\bx}) = P(t_0)$ for $P\in\{1,x,\ldots,x^{d-1}\}$, which is a basis for $\Pi_{d-1}$.
	
	Since the matrix of the system is a Vandermonde matrix, the system is compatible with $N-d$ degrees of freedom. Clearly, the solutions of a linear system form an affine space.
\end{proof}

\begin{rmk}
    Since the coefficients $\ba$ can be obtained by solving the above linear system, where $t_0$ and its powers only appear in the right part, then $\ba$ is $\cC^\infty$ as a function of $t_0$. Since the associated $\Psi$ is a linear combination with weights $\ba$, then $\Psi(\bff)$ as a function of \(t_0\) is $\cC^\infty$ too.
\end{rmk}

Now we see that this affine space can be easily parameterized. Let us denote $\Pi_d|_{\bx} = \{ P|_{\bx} \ : \ P\in\Pi_d\}$ which is a linear subspace of $\R^N$ of dimension $d+1$. Then, we can consider a full-rank linear operator $\nabla_d:\R^N \lra \R^{N-d}$ projecting any vector in $\R^N$ to the orthogonal complement of $\Pi_{d-1}|_{\bx}$, which is isomorphic to $\R^{N-d}$. By definition, $\nabla_d$ is such that
\begin{equation} \label{eq:annihilation}
	\Pi_{d-1}|_{\bx} = \ker(\dif).
\end{equation}
According to \cite{dynlevinluzzato,Lopez24}, $\nabla_d$ is an \emph{annihilation operator} of $\Pi_{d-1}$. We denote it by $\nabla_d$ because, for equal-spaced grids $\bx$, the finite difference operator of order $d$ can be taken. In that case, its matrix representation is a Toeplitz matrix whose non-zero entries are the combinatorial numbers $(-1)^i \binom{d}{i}$, $0\leq i \leq d$.

\begin{prop} \label{prop:affine_space} 
	The affine space of approximants reproducing $\Pi_{d-1}$ with the setting $(\bx, t_0)$ is
	\begin{equation} \label{eq:affine_space_matrix}
		\left \{ \ba^0 + \dif^T \bbeta : \bbeta\in\R^{N-d} \right \},
	\end{equation}
	where $\ba^0\in\R^N$ are the coefficients of some approximant reproducing $\Pi_{d-1}$ with the setting $(\bx, t_0)$.
\end{prop}
\begin{proof}
	According to Lemma \ref{lema:reproduction}, this affine space has dimension $N-d$. Since \eqref{eq:affine_space_matrix} is also affine of such dimensionality, we only have to check that one is contained in the other. In particular, we check that any approximant $\ba^0 + \dif^T \bbeta$ reproduces $\Pi_{d-1}$:
$
	\left ( \ba^0 + \dif^T \bbeta \right )^T P|_{\bx} =
	(\ba^0)^T P|_{\bx} + \beta^T  \underbrace{\dif P|_{\bx}}_{0} =
	P(t_0).
$
\end{proof}

Using the parametrization \eqref{eq:affine_space_matrix}, the optimization problem \eqref{eq:constrained_optimization} can be formulated without constraints as in \eqref{eq:unconstrained_optimization}.
\begin{thm} \label{thm:existance_of_the_best}
    Suppose that $\hat \Omega$ is a definite positive and denote its Choleski decomposition by $\hat \Omega = \Omega^T \Omega$.
	The problem
	\begin{equation} \label{eq:unconstrained_optimization}
		\min \left \{ \|\Omega(\ba^0 + \dif^T \bbeta) \|_2^2 : \bbeta\in\R^{N-d} \right \}
	\end{equation}
	has a unique solution, $\bbeta^*$.
	The minimum-variance approximant is
$$ \ba^* = \ba^0 + \dif^T \bbeta^* = (I_{N\times N} - \dif^T (\dif \hat \Omega \dif^T)^{-1} \dif \hat \Omega ) \ba^0,$$
	being $I_{N\times N}$ the identity matrix.
\end{thm}
\begin{proof}
	Due to Proposition \ref{prop:affine_space}, $\Psi$ reproduces $\Pi_{d-1}$ if and only if $\Psi$ belongs to \eqref{eq:affine_space_matrix}. Then \eqref{eq:constrained_optimization} is equivalent to \eqref{eq:unconstrained_optimization}. For solving the last one, the Moore–Penrose inverse can be used, which leads to the solution
$$ \bbeta^* = - (\dif \hat \Omega \dif^T)^{-1} \dif \hat \Omega \ba^0.$$
	Observe that the solution exists and is unique since $\dif$ is full-rank and $\hat \Omega$ is positive definite (thus, invertible).
\end{proof}


\begin{rmk}
	According to the last result, $\ba^*$ is unique. Thus, its expression is independent of the choice of $\ba^0$, but it depends on the noise distribution through $\hat \Omega$, the polynomial degree $d$, the grid $\bx$ and the evaluation point $t_0$.
\end{rmk}

An interesting property of $\ba^*$ is that its values are related to the point evaluations of some $\Pi_{d-1}$ polynomial, since
\begin{equation} \label{eq_hatOmega_a_in_Ker}
	\dif \hat \Omega \ba^* = (\dif \hat \Omega - \dif \hat \Omega \dif^T (\dif \hat \Omega \dif^T)^{-1} \dif \hat \Omega) \ba^0 = (\dif\hat \Omega - \dif\hat \Omega) \ba^0 = 0.
\end{equation}
Therefore, we can introduce the following result:
\begin{cor} \label{cor:rule_is_poly}
	There exists $Q\in\Pi_{d-1}$ such that $ \hat \Omega \ba^* = Q|_{\bx}$.
	If the noise is uniform and uncorrelated, so $\hat \Omega$ is a constant diagonal matrix, then it holds true that $ \ba^* = Q|_{\bx}$.
\end{cor}
\begin{proof}
	By \eqref{eq:annihilation} and \eqref{eq_hatOmega_a_in_Ker}, $\hat \Omega \ba^* \in \ker(\nabla_d) = \Pi_d|_{\bx}$. Then, there exists $Q\in\Pi_{d-1}$ such that $ \hat \Omega \ba^* = Q|_{\bx}$.
\end{proof}

We saw that if $\ba^*$ is minimum-variance, then $\hat \Omega \ba^* \in \ker(\nabla_d)$. The opposite is also true: if $\hat \Omega \ba \in \ker(\nabla_d)$ and $\ba$ reproduces $\Pi_{d-1}$, then $\ba$ is minimum-variance. We state this in the next result.

\begin{thm} \label{thm:caracterization_of_optimal}
	Let $\Psi$ be an approximant with coefficients $\ba$ that reproduces $\Pi_{d-1}$. It is minimum-variance if and only if $\exists Q\in\Pi_{d-1}$ such that $\ba = \hat \Omega^{-1} Q|_{\bx}$.
\end{thm}
\begin{proof}
	If it is minimum-variance, Corollary \ref{cor:rule_is_poly} implies that $\hat \Omega \ba \in \ker(\nabla_d)$. Let us prove the opposite. Let $\bbeta$ be such that $\ba = \ba^0 + \dif^T \bbeta$. If $\hat \Omega \ba \in \ker(\nabla_d)$, then $\bbeta$ fulfills the linear system $\dif \hat \Omega ( \ba^0 + \dif^T \bbeta )=0.$ The system has a unique solution, since $\dif \hat \Omega \dif^T$ is invertible, as observed in the proof of Theorem \ref{thm:existance_of_the_best}. The solution is $\bbeta = - (\dif \hat \Omega \dif^T)^{-1} \dif \hat \Omega \ba^0 = \bbeta^*$. Then, $\ba = \ba^*$ is minimum-variance.
\end{proof}

This result suggests a form of characterization: consider an arbitrary approximant that reproduces $\Pi_{d-1}$. If there exists a positive definite matrix $\hat \Omega$ such that $\hat \Omega \ba$ corresponds to point evaluations of a polynomial of degree at most $d-1$, then the approximant is optimal for refining noisy data whose noise distribution has $\hat \Omega$ as covariance matrix. Note that the choice of $\hat \Omega$ is not unique, implying that an approximant $\ba$ can be minimum-variance for multiple noise covariance matrices. A simple example is that $\ba$ remains minimum-variance for $\lambda \hat \Omega$, for any $\lambda\in\R\setminus\{0\}$.

Finding the minimum-variance approximant as proposed in Theorem \ref{thm:existance_of_the_best} requires solving a $(N-d)\times (N-d)$ linear system whose matrix is $\dif \hat \Omega \dif^T$. But $N$ may be large in practice. An alternative approach is suggested by Theorem \ref{thm:caracterization_of_optimal}: Find the $d$ coefficients of some $Q\in\Pi_{d-1}$ such that $\Psi(\bff):= \sum_{i=1}^{N} (\hat \Omega^{-1} Q|_{\bx})_i f_i$ reproduces $\Pi_{d-1}$, which means that
\begin{equation} \label{eq:d_times_d}
    \sum_{i=1}^{N} (\hat \Omega^{-1} Q|_{\bx})_i x_i^s = t_0^s, \qquad s = 0,1,\ldots,d-1.
\end{equation}
This is a $d\times d$ linear system, and $d$ is usually small in practice.
A double track for computing the approximant, similar to this one, was also proposed in \cite{CR76}, but with a different description which does not use annihilation operators.

Inspired by Theorem 2 of \cite{CR76}, we proposed the next more explicit expression of the minimum-variance approximant.
\begin{thm}\label{thm:orthonormal}
    Consider the inner product $(P,Q) = P|_{\bx}^T \hat\Omega^{-1} Q|_{\bx}$ for the space $\Pi_{N-1}$.
    Let $\{P^s\}_{s=0}^{d-1}$ be a family of orthonormal polynomials such that $P^s\in\Pi_s$. The coefficients of the minimum-variance approximant for $\hat\Omega$, $t_0$, $\bx$ and $d-1$ are
    \[
        \ba = \hat\Omega^{-1} Q|_{\bx}, \quad Q:=\sum_{j=0}^{d-1} P^j(t_0) P^j.
    \]
\end{thm}
\begin{proof}
    By applying Theorem \ref{thm:caracterization_of_optimal}, it remains to show that with such $Q$ the approximant certainly reproduces $\Pi_{d-1}$. That is, we may show that the approximant is exact for a basis of $\Pi_{d-1}$, as done in \eqref{eq:d_times_d}. In particular, for the basis $\{P^s\}_{s=0}^{d-1}$, we may check that
    \[
    \sum_{i=1}^{N} (\hat \Omega^{-1} Q|_{\bx})_i P^s(x_i) = P^s(t_0), \qquad s = 0,1,\ldots,d-1.
    \]
    Indeed, for $Q=\sum_{j=0}^{d-1} P^j(t_0) P^j$,
    \begin{align*}
        \sum_{i=1}^{N} (\hat \Omega^{-1} Q|_{\bx})_i P^s(x_i)
        = (Q|_{\bx})^T \hat \Omega^{-1} P^s|_{\bx}
        = \left(Q,P^s\right)
        = \sum_{j=0}^{d-1} P^j(t_0) \left(P^j,P^s\right)
        = P^s(t_0).
    \end{align*}
\end{proof}

According to \eqref{eq:error_variance}, a lower value of $\|\Omega \ba\|_2$ corresponds to a greater noise reduction. Given the approach adopted in this paper, which focuses on solving a minimization problem, the following statement can be easily proven.
\begin{thm} \label{thm:order}
	Let $d, d',N,N' \in \mathds{Z}$ be such that $0\leq d \leq d'$ and $1\leq N' \leq N$. Let $\bx\subset\R^N, \bx'\subset\R^{N'}$ be two grids such that $\bx' \subset \bx$. Suppose $t_0 \in [x_1, x_N]$, and let $\hat{\Omega} = \Omega^T \Omega$ represent the covariance matrix associated with the noise on the larger grid $\bx$. Denote $\ba^{d,\bx}$, $\ba^{d',\bx'}\in\R^N$ the minimum-variance approximants for their respective grids and polynomial degrees, where $\ba^{d',\bx'}$ is zero-padded at the points $\bx \setminus \bx'$.
	It follows that the approximant with coefficients $\ba^{d,\bx}$ achieves a better denoising performance compared to that with coefficients $\ba^{d',\bx'}$, satisfying the inequality
	\begin{equation} \label{eq:inequality}
		\|\Omega \ba^{d,\bx}\|_2 \leq \|\Omega \ba^{d',\bx'}\|_2.
	\end{equation}
\end{thm}
\begin{proof}
	Since $\ba^{d,\bx}$ solves the minimization problem in \eqref{eq:constrained_optimization}, and $\ba^{d',\bx'}$ satisfies the constraints defined in \eqref{eq:constrained_optimization}, the inequality in \eqref{eq:inequality} follows.
\end{proof}

\begin{cor}
	For any approximant (with coefficients $\ba$) reproducing $\Pi_0$, its denoising capability is lower bounded by $\|\Omega\ba\|_2^2 \geq (\mathbf{1}^T\hat\Omega^{-1}\mathbf{1})^{-\frac12}$,
	where $\mathbf{1} = (1,\ldots,1)^T \in\R^N$.
\end{cor}
\begin{proof}
	As a consequence of Theorem \ref{thm:order}, the best denoising capability is obtained for $d=0$. By Theorem \ref{thm:caracterization_of_optimal}, $\ba^{0,\bx} = \hat\Omega^{-1}(c \mathbf{1})$ for some $c\in\R$. By the reproduction of $\Pi_0$, $\mathbf{1}^T\ba^{0,\bx} = 1$, which implies that $c=(\mathbf{1}^T\hat\Omega^{-1}\mathbf{1})^{-1}$. Then $\|\Omega\ba^{0,\bx}\|^2_2 = (\ba^{0,\bx})^T\hat\Omega\ba^{0,\bx} = (\ba^{0,\bx})^T (c \mathbf{1}) = c$.
\end{proof}

\subsection{Connection with previous research}\label{resprevios}

The linear approximants proposed above can be used to define subdivision schemes. More precisely, given two integers $n,d\geq 0$ and a sequence of bi-infinite univariate grids $(\bx^k)_{k=0}^\infty$, $\bx^k = (x^k_i)_{i\in\Z}$ (for instance, $\bx^k = 2^{-k}\Z$), let us denote $\bx^k_{i-n+1:i+n} := (x^k_{i-n+1},\dots,x^k_{i+n})$ and $\Psi_{(\bx^k_{i-n+1:i+n},t_0)}:\R^N\to \R$, with $N:=2n$, the minimum-variance approximant for the setting $(\bx^k_{i-n+1:i+n},t_0)$ and some given positive-definite matrix $\hat \Omega^{k,i}\in\R^{N\times N}$.
A binary subdivision scheme consists of the iterative application of the sequence of subdivision operators $S^k:\ell(\Z)\to\ell(\Z)$, defined using the proposed approximants as \emph{refinement rules}:
\begin{equation} \label{eq:scheme_definition}
    (S^k (\bff))_{2i} := \Psi_{(\bx^k_{i-n+1:i+n},x^{k+1}_{2i})}(f_{i-n+1},\ldots,f_{i+n}),\quad    (S^k (\bff))_{2i+1} := \Psi_{(\bx^k_{i-n+1:i+n},x^{k+1}_{2i+1})}(f_{i-n+1},\ldots,f_{i+n}), \qquad \forall i\in\Z,
\end{equation}
where $\bff = (f_i)_{i\in\Z}$.

We are going to demonstrate that our approximants extend those subdivision rules presented in \cite{DHHS15} and, more broadly, in \cite{AY13,LUYA}. The subdivision rules in these studies are derived by solving a weighted regression problem using $\Pi_{d-1}$ polynomials, which, by construction, reproduce $\Pi_{d-1}$. According to Theorem 3.4 in \cite{LUYA}, the coefficients of these rules are expressed as $\ba = W Q|{\bx}$, where $Q \in \Pi_{d-1}$, $\bx$ represents a centered equispaced grid, and $W$ is a positive diagonal matrix. As established in Theorem \ref{thm:caracterization_of_optimal}, these rules are optimal to refine noisy data with noise characterized by the covariance matrix $\hat \Omega = W^{-1}$.

The diagonal structure of $W^{-1}$ implies that the authors in \cite{ AY13, DHHS15,LUYA} (unintentionally) focused on the optimal subdivision rules for non-uniform, uncorrelated noise. By contrast, the rules we propose here generalize this framework to optimally handle correlated noise.

In \cite{LUYA}, it is shown that the choice of $W$ significantly affects the approximation capability of the subdivision scheme. Moreover, the subdivision scheme \eqref{eq:scheme_definition} defined with $\bx^k = 2^{-k}\Z$ was proven to be convergent for $d = 0, 1$ for any positive diagonal matrix $W$, and for $d = 2, 3$ under specific choices of $W$.

Similarly, comparing with the MLS method, we observe that it also fits our description. The MLS coefficients can be obtained by minimizing \eqref{eq:constrained_optimization} with $\hat\Omega$ a diagonal matrix (in the usual description of MLS, as in \cite{Levin98}) or positive definite $\hat\Omega$ (in the more general case, as in \cite{BS89}). That is, not only MLS is Backus-Gilbert optimal, but also it is minimum-variance for a certain noise distribution.

The approximation of linear functionals ($L(f) = f(t_0)$, in our study) by using linear approximants, such that minimize the variance related to the uncertainty in the data, is a classic topic studied, for example, in \cite{CR76}. In that paper a general linear functional $L$ is considered, but the uncertainty in the data was assumed to be independent and identically distributed, in contrast to ours. Firstly, the authors introduced different linear systems to compute the approximant, similar to our proposals in Theorem \ref{thm:existance_of_the_best} (equivalent to solving a $(N-d)\times(N-d)$ system) and in Equation \eqref{eq:d_times_d} (a $d\times d$ system). The novelty in our work relies on the use of annihilation operators and polynomial expression of the coefficients to derive these systems, as well extending the more general case where noise has a non-diagonal covariance matrix. A worth-mentioning result in \cite{CR76} is a clean expression of the coefficients in terms of $L$ and a family of orthonormal polynomials $\{P^i\}_{i=0}^d$, $P^i\in\Pi_i$:
\[
    a_i = \sum_{s=0}^d L(P^s) P^s(x_i), \qquad i=1,\ldots,N,
\]
which in our case for $\hat\Omega=I_{N\times N}$ is
\(
    a_i^{d,\bx} = \sum_{s=0}^d P^s(t_0) P^s(x_i).
\)
We have successfully generalized this result to a positive definite $\hat\Omega$ in Theorem \ref{thm:orthonormal}, by considering a different inner product from the one considered in \cite{CR76} and using the characterization given in Theorem \ref{thm:caracterization_of_optimal}.

\section{Numerical experiments}\label{numericalexp}

We remark that an extensive experimental validation of minimum-variance approximants is not the aim here because their equivalence to generalized least squares (GLS) is well known, and the topic has been thoroughly studied in the statistical literature. The numerical experiments included in the following are, therefore, intended for completeness and illustration.

In this section, we show the advantages of employing minimum-variance approximants for refining noisy data in three experiments: In the first, we consider uncorrelated noise with varying variances. In the second, we analyze correlated noise with equal variances. The key distinction between these experiments lies in the covariance matrix, denoted $\hat\Omega$. In both cases, we compare the new approximants with the minimum-variance approximant for uniform uncorrelated noise, characterized by coefficients $\ba^0 = (a^0_i)_{i=1}^N$, $a^0_i = \frac1N$, which was studied in \cite{DHHS15,LUYA}. In the final experiment, we provide a visual comparison of these approximants on noisy star-shaped data.

For the  first two experiments, $\bx = (-7,\ldots,7,8)\in\R^{16}$. We analyze the approximants derived for $t_0 = 0,1/4,1/2$ (commonly used as binary subdivision rules) and the polynomial degrees $d' := d-1 = 0,1,2,3$. Let $\ba^{t,d,\hat \Omega}$ denote the minimum-variance approximant for given parameters $t_0,d,\hat \Omega$. To evaluate performance, we use the ratio
\begin{equation} \label{eq:rho}
    \rho^{t_0,d,\hat \Omega}:=\|\Omega \ba^{t_0,d,\hat \Omega}\|_2^2 / \|\Omega \ba^0\|_2^2 \leq 1,
\end{equation}
which quantifies the variance reduction achieved by the minimum-variance approximant compared to $\ba^0$. The results presented in Figure \ref{fig_expermient} demonstrate that the same conclusions hold across all tested values of $t_0$ and $d$. Specifically, the minimum-variance approximant consistently outperforms $\ba^0$, and the performance ratio $\rho^{t_0,d,\hat \Omega}$ can be arbitrarily small.

\begin{paragraph}{First Experiment: Uncorrelated Noise with Varying Variances}
	Given $\epsilon\in(0,1)$, let $\hat\Omega^\epsilon\in\R^{16\times16}$ be a diagonal matrix with entries $\hat\Omega^\epsilon_{i,i} = \epsilon$ for $x_i \leq 0$ and $\hat\Omega^\epsilon_{i,i} = 1$ otherwise. Here, the noise is uncorrelated ($\hat\Omega^\epsilon$ is diagonal), and the noise variance is almost zero in the first half of the data. Figure \ref{fig_expermient}-left illustrates that $\lim_{\epsilon\to0} \rho^{t_0,d,\hat \Omega^\epsilon} = 0$, showing that the minimum-variance approximant leverages the data regions with negligible noise to achieve better performance.
    This highlights the effectiveness of the minimum-variance approximant in scenarios where the noise levels vary significantly across the data.
\end{paragraph}

\begin{paragraph}{Second Experiment: Correlated Noise with Uniform Variance}
	For $\epsilon\in(0,0.1)$, let $\hat\Omega^\epsilon\in\R^{16\times16}$ be a block-diagonal matrix with blocks of the form
$$\begin{pmatrix} 1 & -1+\epsilon & 0 & 0 \\ -1+\epsilon & 1 & -\epsilon & 0 \\ 0 & -\epsilon & 1 & -\epsilon \\ 0 & 0 & -\epsilon & 1 \end{pmatrix}.$$
	In this case, the noise variance equals 1 across all data points, but there are correlations within the grid. As shown in Figure \ref{fig_expermient}-right, $\lim_{\epsilon\to0} \rho^{t_0,d,\hat \Omega^\epsilon} = 0$. This result demonstrates that the minimum-variance approximant effectively exploits noise correlation to enhance denoising performance.
\end{paragraph}

\begin{figure}[!h]
	\centering
	\includegraphics[width=0.45\textwidth,clip]{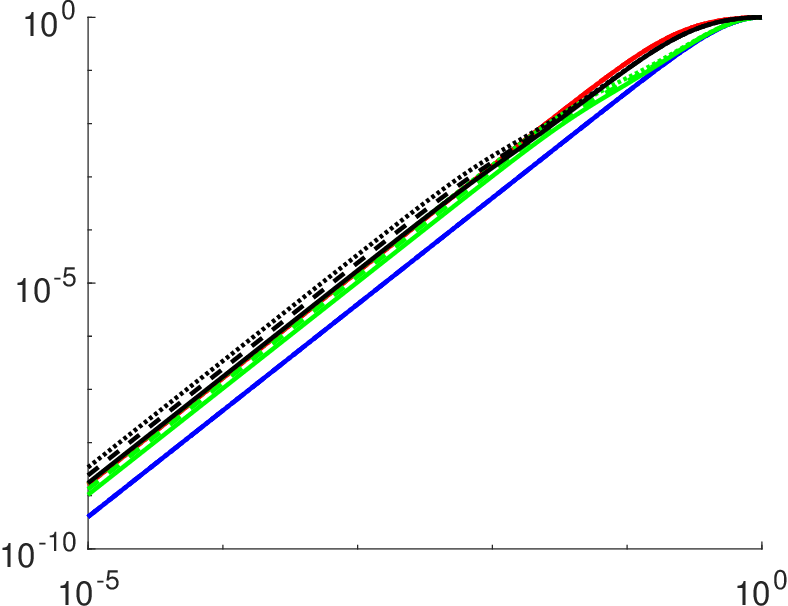}
	\hspace*{20pt}
	\includegraphics[width=0.45\textwidth,clip]{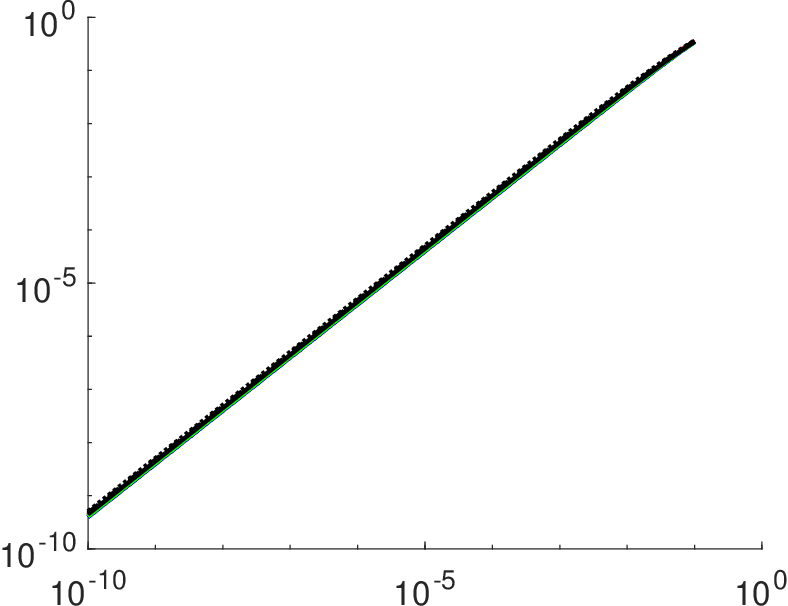}
	\caption{\label{fig_expermient} 
	Graph of $\rho$, defined in \eqref{eq:rho}, as a function of $\epsilon$ for the first experiment (left) and the second experiment (right). Several values of $d$ and $t_0$ are considered. Different line styles are used for $t_0 = 0,1/4,1/2$ (continuous, dashed, and dotted, respectively) and different colors for $d = 0,1,2,3$ (blue, red, green, black, respectively).}
\end{figure}

\begin{paragraph}{Third Experiment: Noisy Star-Shaped Curve}
	Similarly to \cite{DHHS15,LUYA}, we sample $F(t) = (4\cos(t)+\cos(4t), 4\sin(t)-\sin(4t))$, at $t_i = \frac{2\pi i}{320}$, $i=0,\ldots,319$. Noise is added to the two coordinates, with a covariance matrix that is block diagonal. Each block is given by $\frac12 \hat \Omega^{10^{-10}}$ of either the first or the second experiment. We consider a moving evaluation point $t_i$ and stencil $\bx = (t_{i+j})_{j=-7}^8$, taking into account the periodicity of the data, to compute and apply the minimum-variance approximant (as well as $\ba^0$, for comparison) and $d'=1$. The results are shown in Figure \ref{fig_expermient2}, where the minimum-variance approximant outperforms $\ba^0$ in both cases.
\end{paragraph}

\begin{figure}[!h]
	\centering
	\includegraphics[width=0.45\textwidth,clip]{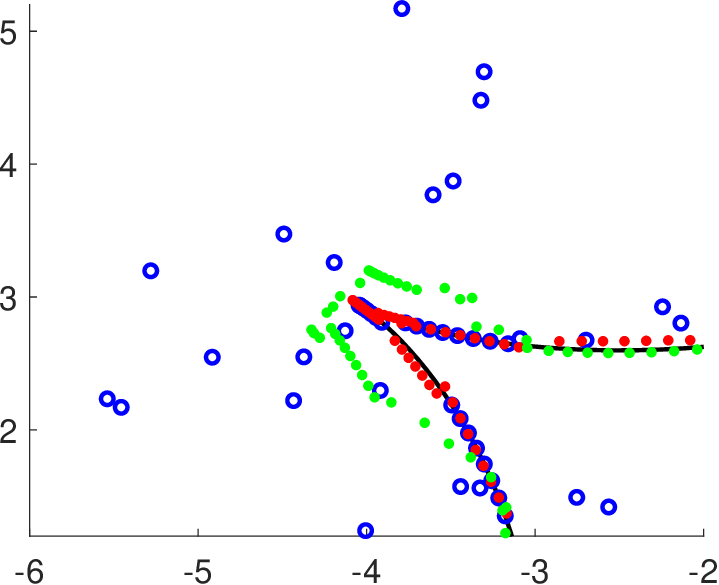}
	\hspace{40pt}
	\includegraphics[width=0.45\textwidth,clip]{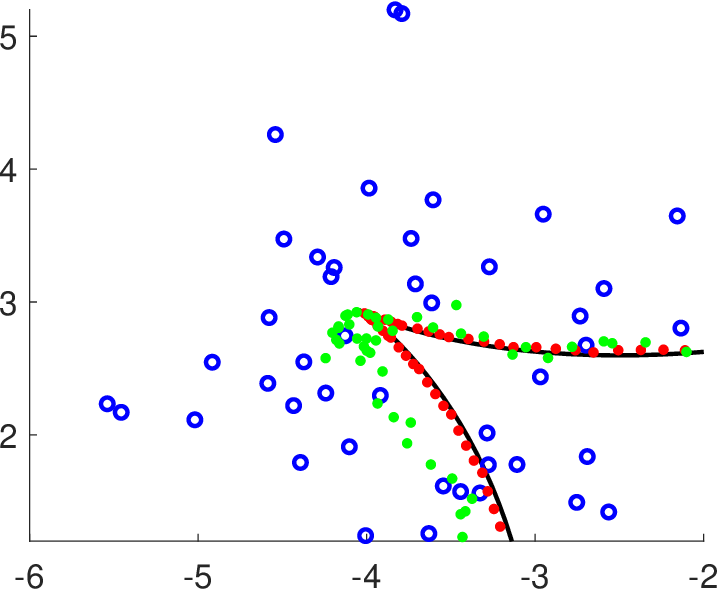}
	\caption{\label{fig_expermient2}
    Third experiment results.
	A star-shaped curve (black line) is sampled and subsequently corrupted with noise (blue circles). The minimum-variance approximants and the one proposed in \cite{DHHS15} are independently applied to refine noisy data (red and green dots, respectively). The left and right plots correspond to the first and second experiments, respectively. A zoom of an interest region is shown.
	}
\end{figure}

\bigskip
The results presented may not be surprising to the reader, since the noise used in the data follows particular distributions. What we aim to show is that constructing approximants as proposed, by minimizing variance, effectively exploits the information provided by the covariance matrix and generates accurate estimates, a much better typical approximation than the more classical ones, which now we know are minimum-variance for uniform uncorrelated noise.

\section{Conclusions and future work}\label{conclusions}

We derived linear approximants that are optimal in the minimum-variance sense. As a novelty, we expressed them in terms of annihilation operators and polynomial evaluations. We contextualized this technique in recent advances in subdivision schemes.

The experiments presented illustrate the effectiveness of linear approximants optimized for noisy data with a known covariance matrix in academic experiments. Such approximants seem particularly advantageous in scenarios where noise exhibits spatially varying magnitudes or significant correlations. In fact, the performance of a minimum-variance approximant can be arbitrarily superior to that of the approximant designed for uniform uncorrelated noise.

We showed in Section \ref{resprevios} how the proposed approximants can be used to define new subdivision schemes. Future work will focus on studying these kinds of subdivision schemes and their performance in denoising when combined with multi-resolution analysis. Since our proposal coincides with GLS, all the advances done in the literature for feasible GLS in the design of new subdivision schemes, also, for the case when the covariance matrix is unknown.


Finally, this work focused on the approximation of the linear functional $L(f)=f(t_0)$, since our intention is to support recent advances in subdivision theory. Nevertheless, we believe that the theory developed here can be easily adapted to a general $L$.

\vspace{-0.25cm}
\section*{Data availability and AI-assistance in the writing process}
\vspace{-0.15cm}

The MATLAB codes used to generate the numerical results presented in this paper are publicly available on GitHub: \url{https://github.com/serlou/minimum-variance-approximants}. These codes enable full replication of all the numerical experiments discussed.
During the preparation of this work, the authors utilized ChatGPT to enhance the clarity and readability of the text. Following this, the authors carefully reviewed and edited the content to ensure accuracy and take full responsibility for the final version of the publication.

\vspace{-0.25cm}
\section*{Acknowledgments and competing interests}
\vspace{-0.15cm}

This research has been supported by the grant PID2023-146836NB-I00 (funded by MCIU/AEI/ 10.13039/501100011033). The authors thank the referees involved in the revision of this work for their insightful and thorough suggestions. The authors declare no conflict of interest.
\vspace{-0.25cm}

\bibliographystyle{plain}

{\small
\begin{thebibliography}{1}
    \bibitem{Aitken36}
    A. C. Aitken.
    \newblock IV.—On least squares and linear combination of observations.
    \newblock {\em Proceedings of the Royal Society of Edinburgh}, 55:42--48, 1936.
	
	\bibitem{AY13}
	F. Ar{\`a}ndiga and D.~F. Y{\'a}{\~n}ez.
	\newblock Generalized wavelets design using kernel methods. application to signal processing.
	\newblock {\em Journal of Computational and Applied Mathematics}, 250:1--15, 2013.

\bibitem{Baltagi}
	B.~H. Baltagi.
	\newblock Econometrics.
	\newblock Springer International Publishing,  2021

    
	\bibitem{BS89}
	L.~P. Bos and K. Salkauskas.
	\newblock Moving least-squares are backus-gilbert optimal.
	\newblock {\em Journal of Approximation Theory}, 59(3):267--275, 1989.

    \bibitem{CR76}
    M. M. Chawla, and T. R. Ramakrishnan.
    \newblock Minimum variance approximate formulas.
    \newblock {\em SIAM Journal on Numerical Analysis}, 13(1):113--128, 1976.
	
	\bibitem{DHHS15}
	N. Dyn, A. Heard, K. Hormann, and N. Sharon.
	\newblock Univariate subdivision schemes for noisy data with geometric applications.
	\newblock {\em Computer Aided Geometric Design}, 37:85--104, 2015.
	
	\bibitem{dynlevinluzzato}
	N. Dyn, D. Levin, and A. Luzzatto.
	\newblock Exponentials reproducing subdivision schemes.
	\newblock {\em Foundations of Computational Mathematics}, 3:187--206, 2003.
	
	\bibitem{Levin98}
	D. Levin.
	\newblock The approximation power of moving least-squares.
	\newblock {\em Mathematics of computation}, 67(224):1517--1531, 1998.
	
	\bibitem{LUYA}
	S. L\'{o}pez-Ure\~{n}a and D.~F. Y\'{a}\~{n}ez.
	\newblock Subdivision schemes based on weighted local polynomial regression: A new technique for the convergence analysis.
	\newblock {\em Journal of Scientific Computing}, 100(1), May 2024.
	
	\bibitem{Lopez24}
	S. L{\'o}pez-Ure{\~n}a.
	\newblock A uniform non-linear subdivision scheme reproducing polynomials at any non-uniform grid.
	\newblock {\em Applied Mathematics and Computation}, 479:128889, 2024.

\bibitem{Maurya22}
	D. Maurya, A. K. Tangirala, and S. Narasimhan.
	\newblock Identification of errors-in-variables ARX models using modified dynamic iterative PCA.
	\newblock {\em Journal of the Franklin Institute}, 359(13):7069--7090, 2022.

\bibitem{Shumway25}
	R. H. Shumway and D. S. Stoffer.
	\newblock {\em Time Series Analysis and Its Applications: With R Examples}.
	\newblock Springer, 2025.
    
	\bibitem{TFM07}
	H. Takeda, S. Farsiu, and P. Milanfar.
	\newblock Kernel regression for image processing and reconstruction.
	\newblock {\em IEEE Transactions on image processing}, 16(2):349--366, 2007.

\bibitem{Yu23}
	X. Yu.
	\newblock Application Research of Spline Interpolation and ARIMA in the Field of Stock Market Forecasting.
	\newblock {\em arXiv preprint arXiv:2311.10759}, 2023.
    
\end{thebibliography}
}

\end{document}